\documentclass[11pt,a4paper, twoside]{article}
\usepackage{paralist}
\usepackage[utf8]{inputenc}
\usepackage{exscale, fullpage,url}
\usepackage[centertags]{amsmath}
\usepackage{color}
\usepackage{amssymb}
\usepackage{amsthm, chngcntr}
\usepackage{dsfont}
\usepackage[all]{xy}
\usepackage{mathrsfs}
\usepackage{upgreek}
\usepackage{tikz-cd}
\usepackage[hidelinks]{hyperref}
\usepackage{graphicx,fp}
\usepackage{mathtools}
\usepackage{calc}
\usepackage{enumitem}

\usepackage[capitalise]{cleveref}
\crefname{enumi}{}{}
\usepackage{stmaryrd} \newif\ifpdf

\usepackage{thmtools}

\numberwithin{equation}{section}
\newtheorem{thm}[equation]{Theorem}

\newcommand{\myendsymbol}{\ensuremath{\diamondsuit}}

\declaretheorem[
  style=definition,
  title=Example,
  qed={$\myendsymbol$},
  refname={example,examples},
  Refname={Example,Examples},
  sharenumber=thm,
]{exa}
\declaretheorem[
  style=definition,
  title=Definition,
  qed={$\myendsymbol$},
  sharenumber=thm,
]{dfn}
\declaretheorem[
  style=plain,
  title=Lemma,
  qed={},
  sharenumber=thm,
]{lem}

\declaretheorem[
  style=plain,
  title=Corollary,
  qed={},
  sharenumber=thm,
]{cor}
\declaretheorem[
  style=definition,
  title=Remark,
  qed={$\myendsymbol$},
  sharenumber=thm,
]{rmk}

\declaretheorem[
  style=definition,
  title=Notation,
  qed={$\myendsymbol$},
  sharenumber=thm,
]{ntn}
\declaretheorem[
  style=definition,
  title=Hypothesis,
  qed={$\myendsymbol$},
  sharenumber=thm,
]{hyp}
\declaretheorem[
  style=definition,
  title=Convention,
  qed={$\myendsymbol$},
  sharenumber=thm,
]{cnv}

\newcommand{\bsx}{\boldsymbol{x}}

\newcommand{\calA}{\mathcal{A}}

\newcommand{\calE}{\mathcal{E}}
\newcommand{\calF}{\mathcal{F}}

\newcommand{\calO}{\mathcal{O}}

\newcommand{\scrM}{\mathscr{M}}

\newcommand{\frakp}{\mathfrak{p}}

\DeclareMathOperator{\Ext}{{\textup{Ext}}}

\DeclareMathOperator{\vol}{\textup{vol}}
\DeclareMathOperator{\Hom}{\textup{Hom}}
\DeclareMathOperator{\sHom}{\mathscr{H}\hspace{-2pt}\mathit{om}}

\DeclareMathOperator{\id}{\textup{id}}

\newcommand{\onto}{\twoheadrightarrow}

\newcommand{\qdeg}{\textup{qdeg}}

\newcommand{\CC}{{\mathbb{C}}}
\newcommand{\DD}{{\mathbb{D}}}

\newcommand{\KK}{{\mathbb{K}}}
\newcommand{\NN}{{\mathbb{N}}}

\newcommand{\RR}{{\mathbb{R}}}
\newcommand{\ZZ}{{\mathbb{Z}}}

\newcommand{\mca}{\mathcal{A}}

\newcommand{\mcd}{\mathcal{D}}

\newcommand{\ra}{\rightarrow}

\newcommand{\into}{\hookrightarrow}
\newcommand{\del}{\partial}

\newsavebox\foobox

\newcommand{\suchthat}{\;\ifnum\currentgrouptype=16 \middle\fi|\;}

\DeclareMathOperator{\sheafHom}{\mathscr{H}\kern -3pt\textit{om}\kern 1pt}
\DeclareMathOperator{\sheafDer}{\mathscr{D}\kern -1pt\textit{er}\kern 1pt}

\DeclareMathOperator{\tdeg}{tdeg}
\DeclareMathOperator{\rank}{rank}
\DeclareMathOperator{\level}{level}

\let\originalleft\left
\let\originalright\right
\renewcommand{\left}{\mathopen{}\mathclose\bgroup\originalleft}
\renewcommand{\right}{\aftergroup\egroup\originalright}

\newcommand{\bolda}{{\mathbf a}}
\newcommand{\boldb}{{\mathbf b}}
\newcommand{\bolde}{{\mathbf e}}
\newcommand{\boldu}{{\mathbf u}}
\newcommand{\boldv}{{\mathbf v}}

\newcommand{\boldx}{{\mathbf x}}
\newcommand{\boldeps}{{\boldsymbol\varepsilon}}

\newcommand{\Hprim}{{H_{\calA,\textrm{prim}}}}

\newcommand{\Iprim}{{I_{\calA,\textrm{prim}}}}
\newcommand{\Iell}{{I_\calA}^{\!\!\!\!\![\ell]}}

\newcommand{\Tprim}{{T_{\mathrm{prim}}}}

\newcommand{\minus}{\smallsetminus}
\DeclareSymbolFont{extraup}{U}{zavm}{m}{n}
\DeclareMathSymbol{\varheartsuit}{\mathalpha}{extraup}{86}
\DeclareMathSymbol{\vardiamond}{\mathalpha}{extraup}{87}

\begin{document}

\title{Hypergeometric systems from groups with torsion}
\author{Thomas Reichelt, Christian Sevenheck, and Uli Walther}

\renewcommand{\thefootnote}{}
\footnotetext{
 CS was partially supported by the DFG grants SE 1114/5-2 and SE 1114/7-1. UW was supported in part by NSF Grant DMS-2100288 and by Simons Foundation Collaboration Grants for Mathematicians
\#580839 and SFI-MPS-TSM-00012928. UW was also supported by NSF Grant DMS-1928930, while in residence at the Simons Laufer Mathematical Sciences Institute (formerly MSRI) in Berkeley, California, during the Spring 2024 semester.
}
\renewcommand{\thefootnote}{\arabic{footnote}}

\newcommand\todo[1]{~\\~\\\noindent TO DO...\\\indent#1~\\...ASAP}

\maketitle

\begin{abstract}
We consider $A$-hypergeometric (or GKZ-)systems in the case where the grading (character) group is an arbitrary finitely generated Abelian group. Emulating the approach taken for classical GKZ-systems in \cite{MMW05} that allows for a coefficient module,
we show that these $D$-modules are holonomic systems. For this purpose we formulate an Euler--Koszul complex in this context, built on an extension of the category of $A$-toric modules. We derive that these new systems are regular holonomic under circumstances that are similar to those that lead to regular holonomic classical GKZ-systems.

For the appropriate coefficient module, our $D$-modules specialize to the "better behaved GKZ-systems" introduced by Borisov and Horja. We certify the corresponding $D$-modules as regular holonomic, and establish a holonomic duality on the level of $D$-modules that was suggested on the level of solutions by Borisov and Horja and later shown by Borisov and Han in a special situation, \cite{BorHor-DualConj,BorEtAl-Dual2}.
\end{abstract}

\tableofcontents

\section{Introduction} \label{sec:intro}

In a series of papers (\cite{BorHor-Mellin, BorHor-MathAnn, BorHor-DualConj, BorEtAl-Dual1, BorEtAl-Dual2}), Borisov together with various collaborators
has introduced a system of linear partial differential equations called the \emph{better behaved GKZ-system}. It has considerable similarities with the members of a class of $D$-modules that generalize the widely studied GKZ-systems, which themselves are far-reaching generalizations of classical hypergeometric differential equations. We refer to our survey \cite{HypSurvey} for a comprehensive overview on GKZ-systems.

A GKZ-system arises from a finite set of elements in a finite-dimensional lattice (the character group), together with a complex parameter vector. In a more general setup, an additional component enters in the form of a coefficient module as introduced in \cite{MMW05}. The systems studied by Borisov \emph{et al.} are special cases in a more general setup. The generalization arises from replacing the lattice by an arbitrary finitely generated Abelian group. This introduction of possible torsion creates interesting new phenomena: for example, the associated toric ideal is now replaced by a smaller ideal that is no longer prime. Within this more general situation, Borisov and his collaborators consider two specific systems. One arises from taking the coefficient module to be the normalization of the associated toric ring, while the other comes from the canonical module of this normalization. In the classical (torsion-free) situation, these choices for the coefficient module assure that the rank of these systems is independent of the parameter vector (which otherwise is not always true, even in the classical case without coefficient module). This absence of rank jumps is why Borisov called these systems ``better behaved''.

In the present paper, we study these systems in general: we allow torsion and arbitrary coefficient modules. We begin with a study of the ideals that replace of toric ideals, and in particular prove that they are still "reasonably close" to the toric ideal. This is necessary to bring in our main tool, the \emph{Euler--Koszul complex} from \cite{MMW05}. We show that these new systems are the terminal homology group of an appropriate Euler--Koszul complex, but in order to allow for torsion in the extended character group, we must extend the category of toric modules (which serve as coefficient modules) from \cite{MMW05} to the category of \emph{twisted toric module}, or \emph{ttoric modules} for short, see Subsection \ref{subsec-toric-mods}.

For this extended category we show that the
 homology groups of the Euler--Koszul complex are holonomic, and we establish a criterion in terms of the coefficient module and the parameter vector to determine (non-)vanishing. As a consequence, we infer that all GKZ-systems with finitely generated Abelian character group and with arbitrary coefficient module are holonomic, and we characterize when those attached to the natural coefficient module are regular holonomic. These statements apply in particular to Borisov's systems.

Following the trail of ideas of \cite{MMW05}, we then investigate the behavior of these systems under holonomic duality. Euler--Koszul functors are self-dual up to a shift in the parameter vector, and thus holonomic duality reduces to standard duality over the underlying ttoric ring. This allows us
to give a compact and conceptional proof of the duality statement in  \cite{BorEtAl-Dual2} regarding the systems considered in \cite{BorHor-MathAnn}, but for an arbitrary character group (allowing torsion) and without the assumption of normality of the underlying toric ring.

\medskip

\thanks{We would like to express our gratitude to the referees for their careful reading and for suggesting further references.}

\section{The $\calA$-graded Category, Toric and Ttoric Modules}

\subsection{General Setup}
For any set $S$, and any ring $\KK$ we write $\KK^S$ for
the free $\KK$-module $\bigoplus_{s\in S}\KK$.

Let
\[
N=F\oplus\ZZ^d
\]
be a finitely generated Abelian group where $F$ is the torsion part of
$N$. The lattice
\[
N^\vee := \Hom_\ZZ(N,\ZZ)
\]
embeds naturally
into $
\Hom_\ZZ(N,\CC)
\simeq N^\vee\otimes_\ZZ\CC
\simeq \CC^d$.  As in \cite{BorHor-Mellin}, let
\[
\mca = \{\bolda_1,\ldots, \bolda_n\} \subseteq N
\]
be a finite multi-set of elements of $N$ (\emph{i.e.}, possibly with repetitions). We set
\[
\ell:=|F|
\]
to be the torsion index,
and choose
\[
\beta \in N \otimes_\ZZ \CC.
\]
There is a morphism
\begin{eqnarray*}
  \pi: N &\ra& N/F=:\bar N =\ZZ^d,\\
  \boldu &\mapsto& \bar\boldu,
\end{eqnarray*}
and whenever it suits us, we identify $\pi(N)$ with its image under the inclusion
\[
\pi(N)\into \pi(N)\otimes_\ZZ\RR=:N_\RR.
\]
We can identify non-canonically
\[
\bar N=\bigoplus_{i=1}^d \ZZ\cdot\boldeps_i
\]
where $\boldeps_1,\ldots,\boldeps_d$ is a chosen basis for $N$.
Denote by
\[
K_\RR:=\RR_{\geq 0}\cdot\pi(\mca)\subseteq N_\RR
\]
the rational polyhedral cone generated by the images $\pi(\bolda_j)\in
N_\RR$.

\begin{hyp}
We shall work throughout under the following assumptions:
\begin{enumerate}
  \item The set $A:=\{\pi(\bolda_j)\mid 1\le j\le n\}$ spans the lattice $\bar N=\pi(N)$ as a $\ZZ$-module; this implies in particular that the index
  \[
  \delta:=[N:\ZZ\calA]
  \]
  is a (possibly strict) divisor of $\ell=|F|$.
  \item The cone $K_\RR$ is \emph{pointed}: the origin is the only unit of the additive semigroup inherited from $\RR^d$ on $K_\RR$.
\end{enumerate}
\end{hyp}
Among the set of all vector space morphisms $\{\tau\colon N_\RR\to\RR\}$
choose a minimal set $\calF_K$ defined such that
\[
K_\RR=\bigcap_{\tau\in\calF_K}(\tau^{-1}([0,\infty))
\]
(minimality means that the intersection is irredundant) and set
\begin{equation}\label{def:SemiGroupK}
K:=\pi^{-1}(K_\RR)\subseteq N,
\end{equation}
the preimage under $\pi$.
Then  $\calF_K$ is finite, and
$\{K_\RR\cap \tau^{-1}(0)\}_{\tau\in\calF_K}$
are the \emph{facets of $\mca$}.  As is traditional, we identify a facet $\tau$ with the intersection $\tau\cap \pi(A)$ as well as with the semigroup spanned by this intersection, or with $\tau\cap K$.

\begin{ntn}
We will further denote by
\[
A:=\pi(\calA),
\]
the finite multi-set consisting of the images in $\bar N=\ZZ^d$ of the elements of $\calA$ under the morphism $\pi$. Fixing a lattice basis $\{\bolde_1,\ldots,\bolde_n\}$ for
\[
L:=\bigoplus \ZZ\cdot \bolde_i \simeq \ZZ^\calA
\]
we will always allow ourselves to view $A$ as an integer $d\times n$ matrix, with columns $\{\bar\bolda_j\}$ and entries $(a_{i,j})$; as $A$ is a multi-set, this matrix might have repeated columns.

We define  polynomial rings, graded by $N$ and $\bar N$ respectively,
\[
R_\calA:=\CC[\{\del_j\mid \bolda_j\in\calA\}],\qquad R_A=\CC[\{\del_j\mid \bar\bolda_j\in\ A\}],
\]
by grading them according to
\[
\deg_\calA(\del_j)=\bolda_j,\qquad \deg_A(\del_j)=\bar\bolda_j.
\]
(If $\calA$ is truly a multi-set, so $\bolda_j=\bolda_{j'}$ for some $j,j'$ with $j\neq j'$ then $R_\calA$ and $R_A$ have one generator each for both copies).

We further introduce the non-commutative  \emph{Weyl algebras}
\[
D_\calA:=R_\calA[\{x_j\mid \bolda_j\in\calA\}], \qquad D_A:=R_A[\{x_j\mid \bar\bolda_j\in\ A\}],
\]
that we grade by
by setting $\deg_{(-)}(x_j)=-\deg_{(-)}(\del_j)$.

In $R_A$ there is the \emph{toric} (prime) \emph{ideal}
\[
I_A:=R_A\cdot\{\del^\boldu-\del^\boldv\mid A\cdot(\boldu-\boldv)=0\}.
\]
Replacing $A$ with $\calA$, we similarly define the $R_\calA$-ideal
\begin{equation}\label{eqn-IcalA}
I_\calA:=R_\calA\cdot\{\del^\boldu-\del^\boldv\mid \calA\cdot(\boldu-\boldv)=0\},
\end{equation}
which we also call \emph{toric};  the dot product here expresses the process of forming $\ZZ$-linear combinations (inside $N$) of elements of $\calA$.

The quotient rings
\[
S_A:=R_A/I_A\quad\text{ and }\quad S_\calA:=R_\calA/I_\calA
\]
are the semigroup rings $S_A=\KK[\NN A]$ and $S_\calA=\KK[\NN\calA]$ respectively.
\end{ntn}

There are graded ring morphisms
\[
\id_R(\pi)\colon R_\calA\to R_A, \qquad
\id_D(\pi)\colon D_\calA\to D_A,
\]
induced by the identity map on the (ungraded) underlying rings, and by the morphism of grading groups $N\to\bar N$. They induce functors (both denoted $\iota_\pi$) from the categories of graded $R_\calA$- and $D_\calA$-modules
to those of graded $R_A$- and $D_A$-modules that are the identity map on the underlying modules.
They are, however, not equivalences of categories, unless $F$ is trivial.

\begin{rmk}
\begin{asparaenum}
\item Note that syzygies between elements of $\calA$ are also syzygies between the corresponding elements of $A$, and so $\iota_\pi(I_{\calA})\subseteq I_A$.
\item
The ideal $I_\calA$ is, in contrast to $I_A$, usually not prime. For example, if $\calA=\{(1\mod 4\ZZ,1),(1\mod 4\ZZ,2)\}\subseteq N:=(\ZZ/4\ZZ)\oplus \ZZ$, then $R_\calA$ and $R_A$ are $\CC[\del_1,\del_2]$, $I_\calA=R_\calA\cdot(\del_1^4-\del_2^2)$, $I_A=R_A\cdot (\del_1^2-\del_2)$, and $\Iell\subseteq R_\calA$ (defined in the proof of Lemma \ref{lem-ttoric} below) is $R_\calA\cdot(\del_1^8-\del_2^4)$.
\end{asparaenum}
\end{rmk}

\subsection{Binomial Ideals}

We review some facts from \cite{EisenbudSturmfels}.
Let
\[
\rho\colon L_\rho\to \CC^*
\]
be a character
on a sublattice $L_\rho$ of $L$.
We say that $L_\rho\into L$ is \emph{saturated} (and then call $\rho$ a \emph{saturated partial character}) if the group $L/L_\rho$ is torsion-free---or, equivalently, if $L_\rho$ is a direct summand of $L$.

\begin{rmk}\label{rmk-extend-rho}
Suppose $\rho'$ is a saturated partial character with associated sublattice $L'$. Let $L''$ be a complement, $L=L'\oplus L''$. Expand $\rho'$ to a character $\rho$ on all of $L$ by setting $\rho(L'')=1$. Then there is a monomial isomorphism $\rho^*$ of $R_A$ given by $\rho^*(\del^{\boldu'+\boldu''})=\rho(\boldu')\cdot \del^{\boldu'+\boldu''}$ for
any $\boldu'\in L', \boldu''\in L''$.
\end{rmk}

From any partial character $\rho$ we define a binomial ideal
\[
I_+(\rho):=R_A\cdot \{\del^\boldu-\rho(\boldu-\boldv)\del^\boldv\mid \boldu,\boldv\in\NN^n, \boldu-\boldv\in  L_\rho\}.
\]
This ideal is prime if and only if $\rho$ is saturated.

Let  now $\frakp$ be an arbitrary binomial prime ideal in $R_A$. There is an induced partition $\{\del_1,\ldots,\del_n\}=\{y_1,\ldots,y_s\}\cup\{z_1,\ldots,z_{n-s}\}$ where $\{y_1,\ldots,y_s\}= \frakp\cap \{\del_1,\ldots,\del_n\}$ are the variables inside $\frakp$. (Note that this is perhaps not the entire linear part of $\frakp$; there might be linear binomials in $\frakp$). There is a corresponding splitting of lattices $\ZZ^n=\ZZ^s\times\ZZ^{n-s}$.
By \cite[Cor.~2.4]{EisenbudSturmfels}, $\frakp$  has the form
\[
\frakp=R_A\cdot\{y_1,\ldots,y_s\}+I_+(\rho)
\]
for some partial saturated character  $\rho$ in the lattice $\ZZ^{n-s}$.

A special case is when $L_\rho=\ker(A)$ and $\rho$ is the trivial character; this results in $I_+(\rho)=I_A$. More generally, if $L_\rho=\ker(A)$ and $\rho$ is an arbitrary character, we denote the prime ideal $I_+(\rho)$ by $I_{A,\rho}$. Note that $I_{A,\rho}$ relates to $I_A$ via the $R_A$-isomorphism from Remark \ref{rmk-extend-rho}.

\medskip

For a face $\tau$ of the cone $K_\RR=\RR_{\geq 0}A$, denote by  $I^\tau_A$ the prime ideal defining the toric ring of $\tau$ as a quotient of $R_A$,
\[
R_A/I^\tau_A\simeq \CC[\tau\cap\NN A].
\]
Then $I^\tau_A$ is the ideal sum of $I_A$ with the ideal generated by all $\del_j$ with $\bar \bolda_j\not\in\tau$.

For a partial character $\rho$ with $L_\rho=\ker(A)$ we define similarly
\[
I^\tau_{A,\rho}:=R_A\cdot(\{\del_j\mid \bar\bolda_j\not\in\tau\}\cup \underbrace{I_+(\rho)}_{=I_{A,\rho}}).
\]
We note that $I^\tau_{A,\rho}$ is a prime ideal. Indeed, $R_A/I^\tau_{A,\rho}=R_B/I_{B,\rho'}$ where $B=A\cap\tau $ corresponds to  the variables that are not generators of $I^\tau_{A,\rho}$, and where $\rho'$ is the restriction of $\rho$ to $\ZZ B$.

We are going to need a statement of the following form.
\begin{lem}\label{lem-Itaurho}
If $\frakp$ is a prime ideal in $R_A$ that is $A$-homogeneous, contains $I_{A,\rho}$ for some character $\rho$ on $\ZZ A$, then it is of the form $I^\tau_{A,\rho}$.
\end{lem}
\begin{proof}
In the $A$-graded module $R_A/\frakp=\bigoplus_{\boldu\in\ZZ A}(R_A/\frakp)_\boldu$, each graded component is either zero or a 1-dimensional vector space. Indeed, already the relations coming from $I_{A,\rho}$ identify (up to nonzero scalars) monomials of $R_A$ if they have equal $A$-degree.
Hence, any $A$-homogeneous element  of $\frakp/I_{A,\rho}$ is the coset of a monomial in $R_A$.  But $\frakp$ being prime implies that
$\frakp$ contains a variable appearing in such monomial.
Let $V_\frakp$ be the ideal generated by the variables that are in $\frakp$ and let $\tau$ be the face of $A$ defined by $[\del_j\in V_\frakp]\Leftrightarrow [\bolda_j\not\in \tau]$; then $\frakp=I^\tau_{A,\rho}$.
\end{proof}

\subsection{Toric Modules}\label{subsec-toric-mods}
We recall from \cite{MMW05} the notion of a \emph{toric $A$-module}, as an $A$-graded $R_A$-module that has a finite filtration by $A$-graded modules such that each composition factor is graded isomorphic to a shifted copy of some $R_A/I_A^\tau$ where $I_A^\tau$ is the toric ideal of one of the faces of the cone spanned by $A$ (the $\tau$ depending on the composition factor). We modify this concept as follows.

\begin{dfn}
An \emph{$\calA$-toric module} is an $\calA$-graded $R_\calA$-module  that has a finite filtration by $\calA$-graded modules such that each composition factor is an $\calA$-graded shifted quotient of  $R_\calA/I_{\calA}$. Note that $\calA$-toric modules are necessarily finitely generated.
\end{dfn}

\medskip

\begin{lem}\label{lem-ttoric}
For any  $\calA$-toric module $M$, the $A$-graded module $\iota_\pi(M)$  has a finite filtration by $A$-graded submodules such that each composition factor is graded isomorphic to a shifted copy of some  $R_A/I^\tau_{A,\rho}$ for suitable $\rho,\tau$ depending on the composition factor.
\end{lem}
\begin{proof}
Recall that $\ell=|F|$ is
the torsion index, and introduce the $R_\calA$-ideal
\[
\Iell:=R_\calA\cdot\{\del^{\ell\boldu}-\del^{\ell\boldv}\mid A\cdot\boldu=A\cdot\boldv\}.
\]
Note that this is an $R_\calA$-ideal but the conditions defining it come from $A$.
Since $\ell \bolda\in 0\oplus\ZZ^d\subseteq F\oplus\ZZ^d$ for any $\bolda\in N$, it follows that if $\bolda$ is a linear relation on $A$ then $\ell\bolda$ is a linear relation on $\calA$. In particular, $\Iell\subseteq I_\calA$, and so $\Iell$ is $\calA$-graded and there is an $A$-graded surjection $\iota_\pi(R_\calA/\Iell)\onto \iota_\pi(R_\calA/I_\calA)$.

By definition, $\iota_\pi(\Iell)$ is an $A$-graded binomial ideal in a
polynomial ring, and by \cite{EisenbudSturmfels} all its
associated prime ideals  $\{\frakp_k\}$ are $A$-graded binomial primes.

We study now the binomial $A$-graded prime ideals $\frakp\subseteq R_A$ containing $\iota_\pi(\Iell)$.
For any $\del^\boldu-\del^\boldv\in I_A$ we have $\del^{\ell\boldu}-\del^{\ell\boldv}=\prod(\del^\boldu-\zeta_\ell^{i}\del^\boldv)\in \iota_\pi(\Iell)$, where $\zeta_\ell$ is a primitive  $\ell$-th root of unity, and where $i$ runs through the elements of $\ZZ/\ell\ZZ$. A prime ideal $\frakp$ containing these products will contain at least one factor from each such product, and so for each binomial in $I_A$ such prime ideal $\frakp$ contains that same binomial twisted by a power of $\zeta_\ell$.
On the other hand, by \cite{EisenbudSturmfels}, $\frakp$ is the sum of a binomial prime ideal $I_+(\rho)$ with an ideal $V_\frakp$ of variables, and we can assume that $I_+(\rho)$ contains no monomial of degree one. Suppose one of the binomials $\del^\boldu-\zeta_\ell^i\del^\boldv\in \frakp$ is not in $R_A\cdot V_\frakp.$ Then
neither of its two monomials can be in $\frakp$ since else they would have to be in $V_\frakp$ as $\frakp$ is prime. Thus, such binomial is in $I_+(\rho)$, which then forces $\boldu-\boldv\in L_\rho$. But then $\del^\boldu-\zeta_\ell^i\del^\boldv$ and $\del^\boldu-\rho(\boldu-\boldv)\del^\boldv$ need to agree up to a constant factor, since both are in $\frakp$ (and hence so are their linear combinations). It follows that for each $\boldu-\boldv\in\ker(A)$, either both $\del^\boldu,\del^\boldv$ are in $V_\frakp$, or $I_+(\rho)$ contains a $\del^\boldu-\zeta_\ell^i\del^\boldv$ with $i$ depending on $\boldu,\boldv$.

Pick an extension $\rho'$ of $\rho$ to $\ZZ^n$, which exists since $\frakp$ is prime and hence $\rho$ is saturated. Then the monomial automorphism  of $R_A$ from Remark \ref{rmk-extend-rho} that sends $\del^\boldu$ to $\rho'(\boldu)\cdot \del^\boldu$ sends $V_\frakp$ to itself and $V_\frakp+I_+(\rho)$ to an $A$-homogeneous prime ideal generated by variables and binomials containing $I_A$. The only such ideals  are, by Lemma \ref{lem-Itaurho} the ideals $I^\tau_A$, and so $\frakp=I^\tau_{A,\rho''}$ for some character $\rho''$ on $L$, and some face $\tau$ of $\NN A$. In particular, the associated primes of $\iota_\pi(R_\calA/\Iell)$ are all of the form $I^\tau_{A,\rho''}$ for suitable $\tau,\rho''$.

\smallskip

Suppose now that $M$ is an arbitrary $\calA$-toric module, and choose one of its composition factors $\tilde M_i:=M_i/M_{i-1}$, an $\calA$-graded shifted quotient of $R_\calA/I_\calA$.
Then $\iota_\pi(\tilde M_i)$ is an $A$-graded quotient of $\iota_\pi(R_\calA/\Iell)$.

It is well-known that any finitely generated ($\bar N$-graded) module $\tilde M$ over any ($\bar N$-graded) Noetherian ring $R$ permits a finite filtration by ($\bar N$-graded) modules such that each composition factor is, up to a shift in the grading, of the form $R/\frakp$ where $\frakp$ is an associated ($\bar N$-graded) prime ideal of $\tilde M$ over $R$. We apply this to $R=\iota_\pi(R_\calA/\Iell)$ and $\tilde M=\iota_\pi(\tilde M_i)$. This means that $\iota_\pi(\tilde  M_i)$ has an $A$-graded filtration such that its composition factors $\tilde M_{i,j}$ are, up to a shift in the grading, of the form $R_A/\frakp_{i,j}$ where each $\frakp_{i,j}$ is an $A$-graded prime containing $\iota_\pi(\Iell)$.
By what we have proved earlier, each $\frakp_{i,j}$ is  thus of the form $I^\tau_{A,\rho}$ for some face of $\NN A$ and some character on $L$, and hence each $\tilde M_{i,j}$ has the claimed property.

It is now standard to assemble a composition chain for $M$ from these composition chains of its composition factors.
\end{proof}

Inspired be this lemma, we make the following definition.
\begin{dfn}
An $A$-graded $R_A$-module $M$ is \emph{$A$-twisted-toric}, or short just \emph{$A$-ttoric}, if it has a finite filtration such that each composition factor is $A$-graded isomorphic to a shifted copy of $R_A/I^\tau_{A,\rho}$ for some face $\tau$ of $A$ and some character $\rho$ on the (saturated) lattice $\RR\tau\cap N$.
\end{dfn}
The image under $\iota_\pi$ of any $\calA$-toric module $M$ is $A$-ttoric by Lemma \ref{lem-ttoric}.

\section{Euler--Koszul Technology}

We review from \cite{MMW05} the following concepts and constructions.

If $m\in M$ is an $A$-homogeneous element of $A$-degree $\boldu=\sum u_i\boldeps_i$ inside an $A$-graded module $M$,  set $\deg_{A,i}(m):=u_i$.
The \emph{true $A$-degrees} $\tdeg_A(M)$ of an $A$-graded module $M=\bigoplus_{\boldu\in \bar N}M_\boldu$ are the collection $\{\boldu\in \bar N\mid M_\boldu\neq 0\}$. As a subset of $\bar N_\CC=\bigoplus \CC\cdot \boldeps_i$, the Zariski closure of $\tdeg_A(M)$ is the set of \emph{$A$-quasi-degrees} $\qdeg_A(M)$. For a toric $A$-module, the quasi-degrees are a finite union of subspaces of the form $\CC\cdot(A\cap \tau)+\boldu$ for faces $\tau$ of $A$ and suitable shift vectors $\boldu\in\bar N$.

For $1\le i\le d$, let
\[
E_i=\sum_{j=1}^n a_{i,j}x_j\del_j\in D_A
\]
be the $i$-th \emph{Euler operator}. Then
$\deg_A(E_i)=0$
and we have the commutator relation
\[
[E_i,P]=-\deg_{A,i}(P)\cdot P
\]
for all $A$-homogeneous $P\in D_A$.

For $M$ any $A$-graded $R_A$-module, and for all $\beta\in L\otimes_\ZZ\CC=\bigoplus_1^d \CC\cdot\boldeps_i$, one can define
an endomorphism (denoted $(E_i-\beta)\bullet$) of left $D_A$-modules on $D_A\otimes_{R_A}M$ by setting
\[
(E_i-\beta)\bullet(P\otimes m):=(E_i+\deg_{A,i}(P)+\deg_{A,i}(m)-\beta_i))(P\otimes m).
\]

If $M=R_A$ with $A$-grading that places $1\in R_A$ into degree zero, this left $D_A$-linear endomorphism on $D_A\otimes_{R_A}R_A=D_A$  is right-multiplication by $E_i-\beta_i$ on $D_A$ as one checks from the commutator relations.

The endomorphisms $(E_i-\beta_i)\bullet$ and $(E_{i'}-\beta_{i'})\bullet$ commute for all $i,i'$ and on all $A$-graded modules $D_A\otimes_{R_A} M$, and thus one can define the Koszul complex on the endomorphisms
\[
\left(
\begin{tikzcd}
0\to D_A\otimes_{R_A}M\arrow{rr}{(E_i-\beta_i)\bullet}&&D_A\otimes_{R_A}M\to 0
\end{tikzcd}\right).
\]

The resulting complex $K_{A,\bullet}(M;\beta)$ (well-defined up to ordering the factors in the tensor product) is the \emph{Euler--Koszul complex} to $M$ and $\beta$ and we view it as positioned in such a way that the terminal module $D_A\otimes M$ of this complex is situated in homological degree zero. We denote the $i$-th homology module by $H_{A,i}(M;\beta)$.

One of the main results in \cite{MMW05} is that every homology group of $K_{A,\bullet}(M,\beta)$ is a holonomic left $D_A$-module, provided that $M$ is an $A$-toric $R_A$-module. We prove next a generalization of this fact to $A$-ttoric modules.

\begin{thm}\label{thm-ttoric-EK}
If $M$ is $A$-ttoric, then for every $\beta\in N_\CC$, the Euler--Koszul complex $K_{A,\bullet}(M,\beta)$ is holonomic (in the sense that each of its homology groups is a holonomic $D_A$-module). Moreover, $K_{A,\bullet}(M;\beta)$ is exact  (has zero homology) if and only if $\beta\notin\qdeg_A(M)$.
\end{thm}
\begin{proof}
We review the cornerstones of the proof in the $A$-toric case from \cite{MMW05} and indicate the necessary changes for the $A$-ttoric case.

If $0\to M'\to M\to M''\to 0$ is an $A$-graded exact sequence of $A$-toric modules, the Euler--Koszul functor induces a long exact sequence of Euler--Koszul homology modules. Basic properties of holonomic modules imply that $H_{A,\bullet}(M;\beta)$ is holonomic if and only if both $H_{A,\bullet}(M',\beta)$ and $H_{A,\bullet}(M'';\beta)$ are holonomic. The definition of $A$-toric modules implies that holonomicity of $K_{A,\bullet}(M;\beta)$ follows from holonomicity of all $K_{A,\bullet}(R_A/I^\tau_A;\beta)$.

 The quasi-degrees of $R_A/I^\tau_A$ are exactly $\CC\cdot(A\cap \tau)$. One shows next that $K_{A,\bullet}(R_A/I^\tau_A;\beta)$ is an exact complex if and only if $\beta$ is not a quasi-degree of $R_A/I^\tau_A$. Since the construction of an Euler--Koszul complex is designed to work well with degree shifts, it follows that $K_{A,\bullet}(M;\beta)$ is exact if and only if $\beta\in\qdeg_A(M)$.

Let now $\beta\in\qdeg_A(R_A/I^\tau_A)=\CC\cdot(A\cap\tau)$.  The main (\emph{i.e.}, terminal) homology group $H_{A,0}(R_A/I^\tau_A;\beta)$ is holonomic since it can be shown to be isomorphic to the $D$-module inverse image under a projection, from $\bar N_\CC$ to a quotient space  of $\bar N_\CC$, of a standard GKZ-system on that quotient space. (More precisely, $H_{A,0}(R_A/I^\tau_A;\beta)$  is of the form $H_{B,0}(R_B/I_B;\gamma)\otimes_\CC\CC[y_1,\ldots,y_k]$ where $B$ is a full rank matrix with the same span as $\tau\cap A$ and $k$ fewer columns than $A$,
compare the proof of \cite[Lemma~4.9]{MMW05}). This then implies that all $H_{A,0}(M;\beta)$ are holonomic for toric $M$. The higher Euler--Koszul homologies of $R_A/I^\tau_A$ are the cosets of $A$-homogeneous elements $m$ in the kernel of the Euler--Koszul complex. Their classes are annihilated by each $(E_i-\beta-\deg(m))\bullet$, and also by a power of $I^\tau_A$ (compare the proof of \cite[Prop.~5.1]{MMW05}). Using a filtration argument, they are a quotient of a finite direct sum of various $H_{A,0}(R_A/I^\tau_A;\gamma)$ and hence holonomic.

Inspecting this proof strategy shows that in order to adapt this proof to the $A$-ttoric case, one just needs to know that $H_{A,0}(R_A/I^\tau_{A,\rho};\beta)$ is holonomic, and that furthermore it is nonzero if and only if $\beta$ is an $A$-quasi-degree of $R_A/I^\tau_{A,\rho}$.

Let $M=R_A/I^\tau_{A,\rho}$. If $\beta$ lies outside $\CC(A\cap\tau)=\qdeg_A(M)$, then the $\CC$-linear span of $E-\beta$ together with  the
products $\{x_j \del_j \mid \bolda_j\notin\tau\}$
is a subset of $D_A\cdot I^\tau_{A,\rho}$ and  contains a nonzero scalar (compare \cite[Lemma~4.9]{MMW05}); hence, $H_{A,0}(M;\beta)=D_A/D_A(I^\tau_{A,\rho},\{E_i-\beta_i\})$ is zero. As in the untwisted case, this implies the vanishing of all higher homology groups. If, on the other hand,  $\beta$ is in the span of $A\cap\tau$ then choose $\dim(\tau)$ many rows of $A$ such that their restriction to the columns in $\tau$ of these rows are linearly independent. Then write $B$ for  the corresponding $\rank(A\cap\tau)\times |A\cap\tau|$-submatrix, denote $\beta^\tau$ the restriction of $\beta$ to the column span of $B$, and
set $E^\tau_i:=\sum_{j\in \tau}a_{i,j}x_j\del_j$.
If now $\rho^B$ is the restriction of $\rho$ to $\ker(B)\subseteq \ker(A)$, it follows that $H_{A,0}(M;\beta)=(D_B/D_B\cdot(I_{B,\rho^B},\{E_i^\tau-\beta_i^\tau\}))\otimes_\CC \CC[\{x_j\mid\bolda_j\notin\tau\}]$. The partial saturated character $\rho^B$ can be used to induce an automorphism of ($R_B$ as in Remark \ref{rmk-extend-rho}, and hence also on) the Weyl algebra $D_B$ that carries $D_B/D_B\cdot(I_{B,\rho^\rho},\{E_i^\tau-\beta_i^\tau\})$ to the classical GKZ-system $H_{B,0}(R_B/I_B;\beta^\tau)$. The latter is, of course, nonzero.
\end{proof}

\begin{rmk}
Holonomicity of $H_{A,0}(R_A/I^\tau_{A,\rho};\beta)$ was already established by Dickenstein, Matusevich and Miller. Indeed,
our setup here  ensures that the underlying binomial ideal is toral in the sense of \cite{DMM-Andean}; see specifically Lemma 3.4 and Theorem 4.5.
\end{rmk}

The \emph{rank} of a $D_A$-module $\scrM$ is the vector space dimension of its classical solution space near a generic point of $\CC^n$. The Cauchy--Kovalevskaya--Kashiwara Theorem shows that this is also the dimension over $\CC(\bsx)$ of the vector space $\CC(\bsx)\otimes_{\CC[\bsx]} \scrM$. For modules such as $H_{A,0}(M;\beta)$ that arise from toric modules $M$, it is shown in \cite{MMW05} that the rank is positive precisely when $\beta$ is a quasi-degree of $M$. This generalizes to the $A$-ttoric case.
\begin{lem}\label{lem-rank-0}
The parameter $\beta$ is a quasi-degree of the $A$-ttoric module $M$ precisely if $H_{A,0}(M;\beta)$ has positive rank. More precisely, the rank function is upper-semicontinuous in $\beta$.
\end{lem}
\begin{proof}
Since rank is additive in short exact sequences, it suffices to prove the statement when $M=R_A/I^\tau_{A,\rho}$ for some face $\tau$ and some character $\rho$ on $\ker(A)$. In that case, the automorphism on $R_A$ arising from Remark \ref{rmk-extend-rho}, and the induced automorphism of $D_A$ (that acts on $x_j$ inversely to the action on $\del_j$) reflects the issue from  $\CC(\bsx)\otimes_{R_A}\scrM$ back to $H_{A,0}(R_A/I^\tau_A;\beta)$, in which case the result is known from work of Gel'fand and his collaborators, and of Adolphson \cite{GGZ87,GKZ89,Adolphson-duke94}.

The proof of the final claim follows the lines of \cite{MMW05} since $A$-ttoric modules give holonomic families.
\end{proof}

\begin{cor}
For any $A$-ttoric module $M$, there is a subspace arrangement $\calE_M$ in $\CC^d$ such that $\beta \notin\calE_M$ implies that the rank of $H_{A,0}(M;\beta)$ is independent of $\beta$.
\end{cor}
\begin{proof}
Following the train of thought in the proof of Lemma \ref{lem-rank-0}, one needs to inspect the statement of the lemma only when $M=R_A/I^\tau_{A,\rho}$. But because of the isomorphism induced via Remark \ref{rmk-extend-rho}, the rank of $H_{A,0}(R_A/I^\tau_{A,\rho};\beta)$ is the same as that of $H_{A,0}(R_A/I_A^\tau;\beta)$, and so one may once again refer to \cite{MMW05}.
\end{proof}

\section{$D$-modules from Modules over Semigroups}

In the next  section we will review the definition of the systems introduced by Borisov and Horja in \cite{BorHor-Mellin, BorHor-MathAnn}, show that they arise as terminal homology groups of an Euler--Koszul complex on a suitable ttoric $A$-module, and deduce finiteness conditions about it. Here, we lay some ground work regarding modules over semigroups.

\subsection{Semigroup Modules}

\begin{cnv}
Throughout, we will assume that our semigroups are subsemigroups of finitely generated abelian groups such that the image in the torsion-free part is pointed.
\end{cnv}
If $S$ is any semigroup, then a set $T$ is an
\emph{$S$-module} if there is an action
\[
\bullet\colon S\times T\to T
\]
that is additive in $S$. We call $T$ \emph{finitely generated (over $S$)} if
$T=\bigcup S\bullet t$ with $t$ running through some finite subset of
$T$. Let $S_+$ be the non-units in $S$ and remark that there are only finitely many units in $S$ (since the torsion is finite and $\pi(S)$ is pointed).

We refer to
\[
\Tprim:=T\minus(S_+\bullet T)
\]
as the \emph{($S$-)primitive elements of $T$}; this is the natural candidate set for a ``collection of minimal generators of $T$ over $S$".

\begin{lem}\label{lem:PrimFinitGen}
Suppose $S$ is a finitely generated semigroup and write $S_+$ for the non-invertible elements of $S$. Suppose $T=\bigcup_1^k S\bullet t_i$ is a finitely generated $S$-module, such that $s\bullet t=t$ only occurs when $s$ is a unit. Then the set $\Tprim:=T\minus S_+\bullet T$ is finite and generates $T$ over $S$.
\end{lem}
\begin{proof}
Let $U$ be the units in $S$; then one may form the factor semigroup $\bar S:=S/U$ which acts on $T$ as well. The identification space $\bar T$ obtained from $T$ by identifying for all $t\in T$ the elements $\{u\bullet t\mid u\in U\}$ is an $\bar S$-module and $\bar s\bullet \bar t=\bar t$ if and only if $\bar s $ is the coset of the identity of $S$. Then $S_+/U=(\bar S)_+$ and cosets of primitive elements in $T$ consist exclusively of primitive elements of $T$, forming primitive elements of $\bar T$. Finiteness of $U$ now implies that for the purpose of the proof we may assume that $U$ is trivial, $S=\bar S$ and $T=\bar T$. Note that this reduction eliminates all torsion from $S$.

In general, let $t\in \Tprim$; according to the setup it is in $\bigcup_1^k S\bullet t_i$ and so there is $i\in\{1,\ldots,k\}$ and $s_i\in S$ with $t=s_i\bullet t_i$. But for a primitive element this can only happen if $s_i$ is a unit (in the reduced setting, the identity,) of $S$.  Since $T$ is assumed to be finitely generated over $S$, we have shown that $|\Tprim|$ is at most $k$ times the number of units of $S$.

We introduce a descending \emph{level} filtration on $T$ by $\level(t)\geq h$ if $t\in \underbrace{(S_++S_++\ldots+S_+)}_{h\text{ copies}}\bullet T$. We say that $t$ is of \emph{height} $h$ if  $\level(t)\geq h$ but $\level(t)\not\geq h+1$. It is convenient to say that every element of $T$ is in level zero, and thus we attach to $t\in T\minus S_+\bullet T$ the height zero.

Let now $g_1,\ldots,g_q$ be nonzero semigroup generators for $S$ and let
$t_1,\ldots,t_k$ be generators for $T$ over $S$. The reduction $S=\bar S$ forces $g_i\in S_+$, and every element of $S_+$ can be written as a sum of the generators (perhaps in many ways).

We claim that every $t\in T$  has well-defined (and, in particular, finite) height. Indeed, suppose $t$ is in every level. So there are equations $t=(\sum_{i=1}^{h'} g_{i,h})\bullet t_{i_h}$ for each $h\in\NN$, with $h'\geq h$ and each $g_{i,h}$ one of our chosen generators in $S_+$.

Since there are only finitely many generators $t_{i_h}$ for $T$, but infinitely many such expressions for $t$, we can choose a subsequence for which always the same $t_{i_h}$ is used, discard the other expressions for $t$, and relabel.

Let $\mu_{h,j}$ be the number of times that the generator $g_j$ appears in the chosen $h$-th sum $t=(\sum_{i=1}^{h'} g_{i,h})\bullet t_{i_h}$.
As there are only finitely many $S$-generators, the set of all $\mu_{h,j}$ cannot be bounded since there are arbitrarily long sum expressions for $t$. Hence, there is at least one generator $g_1$ of $S$ such that the set of natural numbers $\{\mu_{h,1}\}$ is unbounded. Then the  unbounded sequence $(\mu_{h,1})_{h\in\NN}$ contains a strictly increasing subsequence.  Choose such a subsequence, keep the corresponding sum expressions for $t$, discard all others, and relabel.

Now it can happen that for all other generators $g_i\in S_+$ of $S$ the sets $\{\mu_{i,h}\}$ are finite. If so, move to the next paragraph. Otherwise, take a generator $g_2$ that appears with unbounded multiplicity. Repeating the argument in the previous paragraph,  choose a subsequence of expressions for $t$ in which now both $\mu_{h,1}$ and $\mu_{h,2}$ are strictly increasing. Iterate this process, if necessary.

In this manner, we find an infinite sequence of expressions for $t$ in which the multiplicities for $g_1,\ldots,g_r\in S_+$ are strictly increasing, and the multiplicities for $g_{r+1},\ldots,g_q$ are uniformly bounded. The pigeon hole principle now dictates that there are two expressions for $t$ in which for all $g_i$ the multiplicity of $g_i$ in the second expression strictly dominates the multiplicity in the first. So we have
\[
(a_1g_1+\ldots+a_qg_q)\bullet t_i=t=(b_1g_1+\ldots b_qg_q)\bullet t_i,
\]
where $a_j< b_j$ for $j=1,\ldots,q$, and so $\sum_{j=1}^q (b_j-a_j)g_j\in S_+$ acts on $t_i$ as identity. By hypothesis, this implies that this sum is a unit (that is, the identity in the reduced setting).
By contradiction, $t$ must have a finite level, as claimed.

We are now ready to finish the proof of the lemma.
By definition, $\Tprim$ are the elements of height zero. Any generator $t_i$ of $T$ that is not height zero is of the form $s\bullet t$ for some $s\in S_+$ and some $t\in T$. By definition of level and height, at least one such rewriting $t_i=s\bullet t$ exists where the height of $t$ is smaller than the height of $t_i$. Then $\{t_1,\ldots,t_k,t\}\minus\{t_i\}$ is a generating set for $T$ over $S$. By iteration, we can replace any element of positive height among $\{t_1,\ldots,t_k\}$ by one of height zero, while preserving the fact that they are a generating set. It follows that $\Tprim$ generates $T$ over $S$.
\end{proof}

\begin{ntn}
On the semigroups $\NN \calA$ and $\NN A$ there is a tautological degree function, $\deg_A(-)$ with values in $\ZZ^d$ that sends an element to its natural image in $\NN A$.

Suppose $T$ is an $S$-module with $S=\NN A$ or $\NN\calA$. We say $T$ \emph{has an $A$-grading} if there is a map $(-)_A\colon T\to \ZZ A$ such that for all $s\in S,t\in T$ we have $(s\bullet t)_A=\deg_A(s)+(t)_A$. All $S$-modules inside the Grothendieck group $\ZZ S$ have a natural $A$-grading.
\end{ntn}

Define for an $\NN\calA$-module $T$ the $S_\calA$-module
\[
M_T:=\bigoplus_{t\in T}\KK\cdot t
\]
via the rule $(c_s\cdot s)\bullet(c_t\cdot t):=(c_sc_t)\cdot(s\bullet t)$, for $c_s,c_t\in\KK, s\in \NN\calA,t\in T$. If $T$ has an $A$-grading $(-)_A$, then $M_T$ becomes an $A$-graded $R_\calA$-module via $\deg_A(c_t\cdot t):=(t)_A$. If $T$ is finitely generated over $\NN\calA$ then $M_T$ is a finitely generated $A$-graded $S_\calA=\KK[\NN\calA]$-module, and hence $\calA$-toric.

\begin{thm}\label{thm-holonomic}
For every finitely generated $A$-graded $\NN\calA$-module $T$ and all $\beta\in N_\CC$, the $D_A$-module $H_{A,0}(\iota_\pi(M_T);\beta)$ is holonomic. In fact, $H_{A,0}(\iota_\pi(M_T);\beta)$ is the terminal homology group of an Euler--Koszul complex in which every homology group is a holonomic $D_\calA$-module. This Euler--Koszul complex is exact if and only if $\beta$ is not in the Zariski closure of $\deg_A(T)$.
\end{thm}
\begin{proof}
In view of Theorem \ref{thm-ttoric-EK}, all we need to show is that $\iota_\pi(M_T)$ is $A$-ttoric. But that follows from Lemma \ref{lem-ttoric}.
\end{proof}

\begin{cor}\label{cor-regular}
If there is a linear functional $h\colon N\to \RR$ such that $\calA\subseteq h^{-1}(1)$ then for any finitely generated $\NN\calA$-module $T$ and for all $\beta\in\Hom(N,\CC)$, the Euler--Koszul complex on $M_T$ has regular holonomic homology.
\end{cor}
\begin{proof}
In the proof of Theorem \ref{thm-ttoric-EK} it is shown that every homology group of the Euler--Koszul complex on $M_T$ is a quotient of a finite sum of modules that are isomorphic to  $A$-hypergeometric systems in the sense of \cite{MMW05}. But, if a functional $h$ as in the corollary exists, then
every GKZ-system $H_A(\beta)$ is regular holonomic  by \cite{SchulzeWalther-duke}.
 The corollary follows from standard properties of regular holonomicity.
\end{proof}

\subsection{Rank and Duality}

The transposition
$\tau(x^\bolda\del^\boldb)=(-\del)^\boldb x^\bolda$ on $D_\calA$ provides an equivalence of the categories of left and right $D_A$-modules, that corresponds on the level of sheaves to the tensor product with the canonical sheaf $\omega_{\CC^n}$ in one direction, and $\sHom_{\calO_{\CC^n}}(\omega_{\CC^n},-)$   in the other.

The automorphism $x\mapsto -x$ on $\CC^n$ induces, via $\del\mapsto -\del$, an auto-equivalence $(-) \mapsto (-)^-$ on the categories of $\ZZ A$-graded $R_A$-modules and $D_A$-modules (but not on the category of
$\ZZ A$--graded $S_A$-modules, since $I_A$ is in general not  preserved under $(- )^-$). The formation of Euler--Koszul complexes is equivariant under this sign change since $E_i-\beta_i= (E_i-\beta_i)^-$. Moreover, for $\ZZ A$-graded $R_A$-modules $N$ we have $D_A \otimes_{R_A} N\simeq\tau (N \otimes_{R_A} D_A)^-$ as left $D_A$-modules, where the tensor products exploit the two different
$R_A$-structures on $D_A$.

Let $M$ be a  finitely generated $\ZZ A$-graded $R_A$-module and choose a minimal free $\ZZ A$-graded resolution $F_\bullet$. Denoting
\begin{equation}\label{eqn-koszul-free}
\varepsilon_A:=\sum_{\bolda\in \calA}\pi(\bolda)
\end{equation}
the sum of all elements of the multi-set $A$,
it is explained in \cite[Section 6]{MMW05}, that there is a natural identification of the complexes
\[
\tau \Hom_{D_A}(K_{A,\bullet}(F_\bullet;E-\beta),D_A)^- \simeq K_{,\bullet}(\Hom_{D_A}(F_\bullet,D_A);-E-\beta-\varepsilon_A).
\]
The same construction works if $F_\bullet$ is  any
free finite $\ZZ A$-graded complex, and in this setting $F_\bullet$ having toric homology assures the resulting homology groups of the Euler--Koszul complex to be holonomic. In consequence, for a finite $\ZZ A$-graded complex $F_\bullet$ with toric homology, the holonomic dual of the higher Euler--Koszul homology $H_{A,i}(F_\bullet;\beta)$ is, up to a degree shift by $\varepsilon_A$, the Euler--Koszul homology $H_{A,n-i}(\Hom_{D_A}(F_\bullet,D_A);-\beta-\varepsilon_A)^-$. In particular, if $F_\bullet$ has only one homology module $M$, say in cohomological position zero, then $\DD H_{A,i}(M,\beta)=H_{A,n-i}(\Hom_{D_A}(F_\bullet,D_A);-\beta-\varepsilon_A)^-$, and if $M$ is a Cohen--Macaulay $R_A$-module of dimension $d$ then (there is no higher Euler--Koszul homology of $M$, and)
\begin{equation}\label{eq:Duality}
\DD H_{A,0}(M,\beta)=H_{A,0}(\Ext^{n-d}_{R_A}(M,R_A);-\beta-\varepsilon_A)^-.
\end{equation}

We want this isomorphism for finite free $\ZZ A$-graded $D_A$-complexes with $A$-ttoric homology. Inspecting the proof of Thorem 6.3 in \cite{MMW05} reveals that the only property of "toric $F$" that is used is that $A$-toric modules are $A$-graded and produce holonomic Euler--Koszul homology modules. That, however is also the case for $A$-ttoric modules.

\begin{cor}
The duality statement \eqref{eq:Duality} applies in particular to $\calA$-toric modules $M$.\qed
\end{cor}

\section{Applications}

\begin{dfn}[Differential systems on $\calA$-toric modules]
Choose $r\in\NN$, $\beta\in N\otimes_\ZZ\CC$, and let $T\subseteq N^r$ be a module over the semigroup $\NN \calA\subseteq N$. Let
\[
\Phi_T = (\phi_\boldu)_{\boldu \in T}
\]
be an assignment of a (sufficiently differentiable) function in $x_A$ to each element of $T$.
The natural morphism
$N^\vee=\Hom_\ZZ(N,\ZZ)\to \Hom_\CC(N\otimes \CC,\CC)$
allows to define $\mu(\beta)$ for any $\mu\in N^\vee$.

Consider the following system of partial differential equations:
\[
\{\del_{j} \phi_\boldu = \phi_{\boldu +\bolda_j}\}_{\boldu\in T,
\bolda_j\in\calA}\qquad
\cup\qquad
\{\sum_{\bolda_j\in\calA} \mu(\bolda_j) x_j \del_j \phi_\boldu =
\mu(\beta-\boldu) \phi_\boldu\}_{\boldu\in T, \mu\in N^\vee}.
\]
Let us write
\[
E_\mu:=\sum_{\bolda_j\in\calA} \mu(\bolda_j) x_j\del_{j}\quad\in D_\calA.
\]
\end{dfn}
A system of type $\Phi_T$ induces a $D_\calA$-module
\begin{align*}
\Phi_{\mca,0}(T;\beta) :=  D_\calA^T /H_\calA(T;\beta),
\end{align*}
where $H_\calA(T;\beta)\subseteq D^T_\calA$ is the left $D_\calA$-module
\begin{align}\label{eqn-bbgkz-rels}
H_\calA(T;\beta):=D_\calA\cdot\left(\{ \del_{j}1_\boldu -1_{\boldu+\bolda_j}\}_{\boldu\in T,\bolda_j\in\calA} \quad\cup\quad
\{(E_\mu - \mu(\beta-\boldu))
\cdot 1_{\boldu}\}_{\boldu\in T,\mu\in N^\vee}\right)
\end{align}
and $1_\boldu$ is the element of $\mcd^T$  that is  1 in coordinate
$\boldu$ and zero everywhere else.

The system $\Phi_T$ is the space of classical solutions to $\Phi_{\calA}(T;\beta)$. In particular, it can be identified with the localization of $\Phi_{\calA}(T;\beta)$ to a generic point.

Since $\Tprim$ generates $T$
as $\NN \mca$-module by Lemma \ref{lem:PrimFinitGen},  the defining
relations spelled out in \eqref{eqn-bbgkz-rels} imply that
$\Phi_{\calA}(T;\beta)$ is as $D_\mca$-module generated by the
cosets of all $1_\boldu$ with $\boldu\in\Tprim$; it is hence  a finitely
generated $D_\mca$-module and we can replace in the definition of
$\Phi_{\calA,0}(T;\beta)$ the set $T$  by $\Tprim$:
\[
\Phi_{\calA,0}(T;\beta)\,\simeq_{D_\mca}\,
D_\calA^\Tprim / \Hprim(T;\beta)
\]
where
\[
\Hprim(T;\beta):=H_{\calA}(T;\beta)\cap D_\calA^\Tprim.
\]

Note that if we set
\[
\Iprim(T):=D_\calA\cdot\{\del_\calA^\boldu\cdot 1_{\boldu'}-\del_\calA^\boldv\cdot 1_{\boldv'}\mid \boldu+\boldu'=\boldv+\boldv'\}_{\boldu',\boldv'\in\Tprim}
\]
then
\begin{equation}\label{eqn-HAprim}
\Hprim(T;\beta)=D_\calA\cdot(\Iprim(T),\{(E -\mu(\beta-\boldu))\cdot 1_\boldu\}_{\mu\in N^\vee,\,\boldu\in\Tprim}).
\end{equation}

\begin{exa}
One example of a finitely generated $\NN\calA$-module  is given by $K$, as defined in Formula
\eqref{def:SemiGroupK}. Indeed, the image $\pi(K)$ is the saturation semigroup of $\NN A$, the set of all lattice points $\boldu\in\ZZ^d$ such that $\NN \boldu$ meets $\NN A$. That $\pi(K)$ is a finitely generated module over $\NN A$ is Gordan's Lemma (see, \emph{e.g.}, \cite[Proposition 1.2.17]{CoxBook}). That $K$ is finite over $\NN \calA$ then follows from the fact that the fibers of $\pi$ are finite.

A second important finitely generated $\NN\calA$-module is the set $K^\circ$, sitting over the interior points of
$K$,
\begin{equation}\label{eqn-Kpos}
K^\circ=\{\boldu\in K\mid \pi(\boldu)\notin \tau\,\,\forall \tau\in\calF_K\}.
\end{equation}
The shortest argument that $K^\circ$  is finitely generated is slightly
roundabout. Let $\KK$ be any field, $S\subseteq \ZZ^d$ a semigroup without non-trivial units, $T\subseteq \ZZ^d$ an
$S$-module, and consider the semigroup ring $\KK[S]=\bigoplus_S
\KK\cdot s$ and the vector space $\bigoplus_T\KK\cdot t$, which in
natural fashion has a $\KK[S]$-module structure.  If $S$ is finitely
generated by $\tau$ elements, then $S$ is a quotient of $\NN^\tau$ and
so $\KK[S]$ is the quotient of a polynomial ring in $\tau$ variables,
and in particular Noetherian. If then $T$ were an infinitely generated
module over $S$, one would have an infinite ascending chain of
$\KK[S]$-submodules of $\bigoplus_T\KK\cdot t$, which could then not
be Noetherian. Now consider the $\NN A$-module $K^\circ$ and the
corresponding $\KK[\NN A]$-module $\KK[\pi(K^\circ)]$. Since $\pi(K)$ is
a finite $\NN A$-module, $\KK[\pi(K)]$ is module-finite over
$\KK[\NN A]$. But $\KK[\pi(K^\circ)]$ is the canonical module of
the normal Cohen--Macaulay ring $\KK[\pi(K)]$ (by work of Hochster, Danilov and Stanley), and thus certainly module-finite over $\KK[\pi(K)]$. Finiteness of fibers of $\pi$ then dictates that $K^\circ$ is finitely generated over $\NN \calA$.
\end{exa}

\begin{rmk}
It follows from the description \eqref{eqn-HAprim} that for every $\calA$-toric module $T$,
\[
\iota_{\pi}(\Phi_{\calA,0}(T;\beta))=H_{A,0}(M_T;\beta)
\]
is the distinguished (terminal) homology group of the Euler--Koszul complex to $\beta$ on $M_T$.

The case $T=K$ inspired Borisov and Horja to write down the differential system $\Phi_K$ and the corresponding $D_A$-module $\Phi_{A}(K;\beta)$ which they termed the \emph{better behaved GKZ-system}.

In particular, the better behaved GKZ-system to $\beta$ in the torsion-free case $F=0$ is precisely the $A$-hypergeometric module that corresponds to the toric $S_A$-module $M_K=\tilde S_A$, the normalization of $S_A$.
\end{rmk}

\begin{exa}\label{ex:ProdStruct}
Let $N=(\ZZ/2\ZZ)\oplus \ZZ$, and assume $\calA$ is the singleton $\bolda_1=(\bar 1,1)$. Then $\NN\calA=\NN\cdot (\bar 1,1)\subseteq N$ is isomorphic to $\NN=\pi(\NN\calA)$ and so $R_\calA$ and $R_A$ are isomorphic to the polynomial ring $\KK[\del_1]$, with corresponding $\calA$- or $A$-grading. The semigroup $K$ is $(\ZZ/2\ZZ)\oplus \NN$ and its semigroup ring is isomorphic to $(\KK[y_0]/(y_0^2-1))\otimes (\KK[y_1])$.

However, under this isomorphism, $\del_1$ acts as the element $y_0y_1$, and in particular does not act trivially on the $y_0$-factor. Thus, (here and in general) there is an obvious abstract vector space isomorphism between $\KK[K]$ and a degree-$|F|$ ring extension of $\tilde S_A$, but it is less clear (in general) that there is such an isomorphism that is an isomorphism of $R_A$-algebras.
\end{exa}

\begin{cor}\label{cor:Rank}
We consider the $D_A$-modules $H_{A,0}(M_K;\beta)$ and $H_{A,0}(M_{K^\circ};\beta)$. Both have rank equal to
$|F|\cdot\vol(A\cup\{0\})$.
\end{cor}
\begin{proof}
By definition, $K$ is the preimage under $\pi$ of the lattice points in the cone spanned by $A=\pi(\calA)$. In particular, as a semigroup, $K=F\oplus \pi(K)$. As a semigroup ring, $\KK[\pi(\calA)]$ is the saturated semigroup ring to $\KK[\NN A]$, and
\[
\KK[K]\simeq \tilde S_A\otimes_\KK \KK[y_1,\ldots,y_k]/(\{y_i^{\ell_i}-1\}_1^k),
\]
where $F=\bigoplus_1^k (\ZZ/\ell_i\ZZ)$ and $\ell=\prod \ell_i$. While the $R_A$-structure on $\KK[K]$ in the above description is not compatible with the "natural" subring $\tilde S_A=\tilde S_A\otimes 1\into \KK[K]$ (as seen in the above Example \ref{ex:ProdStruct}), we show below that there  \emph{is} such a direct sum decomposition that is compatible with this natural structure; we are interested in it because it certifies $\KK[K]$ as a Cohen--Macaulay $R_A$-module.

We start with considering $R_\calA$-algebra morphisms from $\KK[K]$ to $\tilde S_A$. For $1\le i\le k$, let $\zeta_i$ be a primitive $\ell_i$-th root of unity, and choose $\underline{r}:=(r_i\in\NN\mid 0\le r_i\le \ell_i-1)_{i=1}^k$. Then
\begin{eqnarray*}
\phi_{\underline{r}}\colon \KK[K]&\to& \tilde S_A,\\
\{y_i&\mapsto&(\zeta_i)^{r_i}\}_1^k
\end{eqnarray*}
is a surjective $R_\calA$-algebra morphism. Now set $\phi:=\bigoplus_{\underline{r}}\phi_{\underline{r}}\colon \KK[K]\to \bigoplus_{\underline{r}} \tilde S_A$.

The map $\phi$ is an injective $R_\calA$-module map. Indeed, if $f\in\KK[K]$ is in the kernel of $\phi_{\underline{r}}$, then $f$ is in the ideal generated by $\{y_i-\zeta_i^{r_i}\}_1^k$. Thus, if $f$ is in the kernel of $\phi$ then $f$ is in the ideal generated by $\{y_i^{\ell_i}-1\}_1^k$ and hence zero in $\KK[K]$. But $\phi$ is also surjective, which becomes obvious if one considers one variable at the time: for any $\KK$-algebra $R$, the morphism that sends $(\KK[y]/(y^q-1))\otimes_\KK R$ to $\bigoplus_0^{q-1} R$ by evaluating $y$ to $\zeta_q^i$ on the $i$-th summand is onto since \[
\begin{pmatrix}
1&1&\ldots&1\\
\zeta_q^1&\zeta_q^2&\ldots&\zeta_q^{q-1}\\
\vdots&&&\vdots\\
\zeta_q^{q-1}&\zeta_q^{2q-2}&\cdots&\zeta_q^{(q-1)(q-1)}
\end{pmatrix}
\]
is a non-singular Vandermonde matrix.

It follows that $\KK[K]$ is $R_\calA$-isomorphic to $\bigoplus_{|F|}\tilde S_A$ when $\tilde S_A$ is given its natural $R_\calA$-structure inherited from $\pi(K)$ being the saturation of the semigroup $\pi(\NN\calA)=\NN A$.

Restriction of the map $\phi$ to $\KK[K^\circ]$ shows that $\KK[K^\circ]$ is $R_\calA$-linearly isomorphic to the free sum of $|F|$ copies of $\KK[\pi(K^\circ)]$. The decompositions show that $H_{A,0}(M_T;\beta)$ is isomorphic to $|F|$ copies of $H_{A,0}(M_{\pi(T)};\beta)$, for $T\in \{K,K^\circ\}$.

It is well-known that $H_{A,0}(M_{\pi(K)};\beta)$ has rank equal to the
(simplicial) volume of the convex hull spanned by the origin and $A$, independently of $\beta$.  Moreover,
$\KK[\pi(K^\circ)]$ is a maximal Cohen--Macaulay $S_A$-module,  and so the rank of $H_{A,0}(M_{K^\circ};\beta)$ is independent of $\beta$ as well. It suffices then to consider a generic $\beta$. But the quasi-degrees of $\pi(K)/\pi(K^\circ)$ are contained in the facets of $\calA$, and thus do not contain a generic $\beta$. By Theorem \ref{thm-ttoric-EK}, applied to the long Euler--Koszul homology sequence derived from  $0\to \pi(K^\circ)\to \pi(K)\to\pi(K)/\pi(K^\circ)\to 0$, one concludes that the ranks of
$H_{A,0}(M_{\pi K^\circ};\beta)$ and $H_{A,0}(M_{\pi(K^\circ)};\beta)$ are always equal.
\end{proof}

Finally, we apply the results on duality of Euler--Koszul homology groups to the module $H_{A,0}(M_K;\beta)$.
\begin{thm}
There is a natural non-degenerate pairing
\[
\left(\CC(\boldx)\otimes_{\CC[\bsx]} H_{A,0}(M_K;\beta)\right)\otimes
\left(\CC(\boldx)\otimes_{\CC[\bsx]} H_{A,0}(M_{K^\circ};-\beta-\varepsilon_A)^-\right)\longrightarrow \CC(\boldx).
\]
In particular, if $\Gamma=H^0(\textup{Sol}(H_{A,0}(M_K;\beta)_\xi))$ resp. $\Gamma'=H^0(\textup{Sol}(H_{A,0}(M_{K^\circ};-\beta-\varepsilon_A)_\xi)^-)$ is the space
of (classical) solutions of $H_{A,0}(M_K;\beta)$ resp. $H_{A,0}(M_{K^\circ};-\beta-\varepsilon_A)^-$ near a generic point $\xi$, then we have a non-degenerate pairing
$\Gamma\times\Gamma'\rightarrow \CC$.
\end{thm}
\begin{proof}

It follows from Formula \eqref{eq:Duality} that
$$
\DD H_{A,0}(M_K,\beta)\cong H_{A,0}(M_{K^\circ};-\beta-\varepsilon_A)^-(\varepsilon_A),
$$
since by what has been shown in the proof of Corollary \ref{cor:Rank}, $M_K$ is a Cohen--Macaulay $R_A$-module,
and then its dualizing module is
$M_{K^\circ}$.
Forgetting the graded structure, localize at the generic point. Then we are looking at holonomically dual, smooth $D$-modules. These are hence mutually dual bundles with mutually dual connections, as follows from the Spencer resolution. (See,
\emph{e.g.}, \cite[Lemma A.11]{DS} for the argument in a slightly different but related situation). The theorem follows.
\end{proof}
\begin{rmk}
If $\bar\bolda_1,\ldots,\bar\bolda_n$ lie on an affine hyperplane in $\bar N$, the functor $(-)^-$ is an auto-equivalence of graded $S_A$-modules, since $I_A$ is then projective. If in this case $\beta$ is the zero vector, we obtain a duality result similar to \cite[Theorem 2.4.]{BorEtAl-Dual2}. Formula \eqref{eq:Duality} is more general, since the $D_A$-modules $H_{A,0}(M_K;\beta)$ and $H_{A,0}(M_{K^\circ};-\beta-\varepsilon_A)^-$ cannot be reconstructed from their (classical) solutions, nor from their restrictions to a generic point. On the other hand, holonomic duality of modules always yields a pairing on the level of solutions at a generic point.

It is an interesting project to check whether the concrete construction in \cite{BorEtAl-Dual2} agrees with our functorial pairing since that would provide a
link between the fundamental techniques of \cite{MMW05} and of the
current paper, and applications of hypergeometric systems in mirror symmetry.
\end{rmk}

\bibliographystyle{amsalpha}
\def\cprime{$'$}
\providecommand{\bysame}{\leavevmode\hbox to3em{\hrulefill}\thinspace}
\providecommand{\MR}{\relax\ifhmode\unskip\space\fi MR }
\providecommand{\MRhref}[2]{  \href{http://www.ams.org/mathscinet-getitem?mr=#1}{#2}
}
\providecommand{\href}[2]{#2}

\noindent
Thomas Reichelt\\
Lehrstuhl f\"ur Mathematik VI \\
Institut f\"ur Mathematik\\
Universit\"at Mannheim,
A 5, 6 \\
68131 Mannheim\\
Germany\\
thomas.reichelt@math.uni-mannheim.de\\

\noindent
Christian Sevenheck\\
Fakult\"at f\"ur Mathematik\\
Technische Universit\"at Chemnitz\\
09107 Chemnitz\\
Germany\\
christian.sevenheck@mathematik.tu-chemnitz.de\\

\noindent
Uli Walther\\
Purdue University\\
Dept.\ of Mathematics\\
150 N.\ University St.\\
West Lafayette, IN 47907\\
USA\\
walther@math.purdue.edu\\

\end{document}